\documentclass[11pt]{amsart}

\usepackage[usenames]{color}
\usepackage{enumerate}
\usepackage{amsmath}
\usepackage{amsthm}
\usepackage{amsfonts}
\usepackage{amssymb}
\usepackage{float}
\usepackage[colorlinks=true,
linkcolor=webgreen,
filecolor=webbrown,
citecolor=webgreen]{hyperref}
\usepackage{url}
\usepackage{tikz}
\usetikzlibrary{math,calc}
\usepackage{graphicx}
\usepackage{epstopdf}
\usepackage{vwcol}

\theoremstyle{plain}

\newtheorem*{theorem*}{Theorem}

\newtheorem*{corollary*}{Corollary}

\newtheorem*{lemma*}{Lemma}

\newtheorem*{proposition*}{Proposition}

\theoremstyle{definition}

\newtheorem*{definition*}{Definition}

\newtheorem*{example*}{Example}

\newtheorem*{conjecture*}{Conjecture}

\newtheorem*{problem*}{Problem}

\newtheorem*{todo*}{TO DO}

\newtheorem*{question*}{Question}

\theoremstyle{definition}

\newtheorem*{remark*}{Remark}
\theoremstyle{remark}

\newtheorem*{notation*}{Notation}

\definecolor{webgreen}{rgb}{0,.5,0}
\definecolor{webbrown}{rgb}{.6,0,0}

\title{Random Surfaces with Boundary}

\author{Chaim Even-Zohar}
\address{Chaim Even-Zohar, The Alan Turing Institute, London, NW1 2DB, UK}
\email{chaim@ucdavis.edu}

\author{Michael Farber}
\address{Michael Farber, Queen Mary University of London, E1 4NS, UK}
\email{m.farber@qmul.ac.uk}

\begin{document}

\maketitle

\begin{abstract}
\vspace*{-2em}
A surface with boundary is randomly generated by gluing polygons along some of their sides. We show that its genus and number of boundary components asymptotically follow a bivariate normal distribution.
\vspace*{-1em}
\end{abstract}

\section{Introduction}

A method for randomly generating a closed orientable surface was suggested by Pippenger and Schleich~\cite{pippenger2006topological}, and in an equivalent form by Brooks and Makover~\cite{brooks2004random}. Take an even number $n$ of triangles, and glue their edges in pairs, in an orientation-preserving way. The glued pairs are uniformly selected from all $(3n-1)!!$ possible matchings between the edges. Every closed orientable surface is realized this way. Other polygons may be used instead of triangles.

The resulting surface is connected with high probability. Using tools from the spectral analysis of the symmetric group~\cite{diaconis1981generating,fomin1997number}, Gamburd proved that the distribution of its genus is asymptotically normal~\cite{gamburd2006poisson}. Its mean is $\frac14 n-\frac12\log n$, and the variance is $\frac14\log n$, up to constant order terms, which have been determined in~\cite{fleming2010large}. For a geometric study of such random surfaces, see~\cite{brooks2004random}.

A similar construction arises from a classical work by Harer and Zagier. Take one large polygon with an even number $n$ of edges, and glue pairs at random such that all $(n-1)!!$ matchings are equally likely. The genus of the resulting surface is again asymptotically normal with mean $\frac14 n -\frac12 \log n$ and variance $\frac14\log n$~\cite{harer1986euler, pippenger2006topological, linial2011expected, chmutov2013genus}.

Chmutov and Pittel~\cite{chmutov2016surface} refined these results, and gave unified proofs for both models, using new estimates on characters of the symmetric group by Larsen and Shalev~\cite{larsen2008characters}. 

\medskip
Here we define and study natural extensions of these constructions to oriented surfaces with boundary. This is done by adding to the polygons designated boundary sides that are not to be glued. Thus, instead of taking $n$ triangles as before, we replace $m$ of them with squares that have one unmatched boundary edge. 
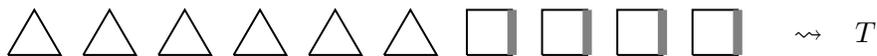
\begin{center}
\begin{tikzpicture}[line width=0.75pt,black,scale=1]
\foreach \x in {0,...,5} \draw (\x,0) -- (\x+0.7,0) -- (\x+0.35,0.6) -- (\x,0);
\foreach \x in {6,...,9} \draw (\x+0.1,0) -- (\x+0.7,0) -- (\x+0.7,0.6) -- (\x+0.1,0.6) -- (\x+0.1,0);
\foreach \x in {6,...,9} \draw[gray,line width=3.5pt] (\x+0.7,0.6) -- (\x+0.7,0);
\pgfresetboundingbox \clip (0,-0.5) rectangle (10,1.0);
\end{tikzpicture}
\raisebox{1.75em}{$\;\;\;\rightsquigarrow\;\;\;T$}\;\;
\end{center}

We glue the $3n$ non-boundary edges as before with a uniformly random matching. The resulting surface is denoted~$T(n,m)$, or~$T$ for short. 

Alternatively we take one polygon, with an even number $n$ of ordinary edges plus $m$ boundary edges. We randomly place the edges of different types around the polygon, where each of the $\tbinom{n+m}{m}$ possibilities is equally likely. Then we randomly match and glue the $n$ non-boundary edges as before. The resulting surface is denoted $S=S(n,m)$. 
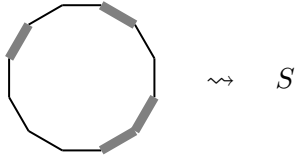
\begin{center}
\begin{tikzpicture}[line width=0.75pt,black,scale=0.5]
\foreach \x in {0,1,3,4,6,7,8,9} \draw (\x*30+15:2) -- (\x*30-15:2);
\foreach \x in {2,5,10,11} \draw[gray,line width=3.5pt] (\x*30+15:2) -- (\x*30-15:2);
\pgfresetboundingbox \clip (-2.2,-2.2) rectangle (2.2,2.2);
\end{tikzpicture}
\raisebox{2.5em}{$\;\;\;\;\rightsquigarrow\;\;\;\;S$}\;\;
\end{center}

We suggest a slight variation on $S(n,m)$, to avoid the redundancy of adjacent boundary edges which are actually equivalent to one. We take a polygon with $n$ edges and randomly select a subset of $m$ corners, uniformly from all $\tbinom{n}{m}$ possible choices. At each of these corners we insert an additional boundary edge, sticking it between the two original ones. We glue the $n$ ordinary edges as before. This random surface will be denoted $S'(n,m)$.

We similarly modify $T(n,m)$, starting with $n$ triangles and then inserting boundary edges at $m$ random corners, picked uniformly from all $\tbinom{3n}{m}$ subsets. The $3n$ original edges are randomly glued as before. This random surface will be denoted $T'(n,m)$.

\medskip
Every compact orientable surface with boundary is obtained with positive probability as $T$ or~$T'$, for some $n$ and $m$ large enough. However, with high probability it is a connected one, exactly as shown for closed surfaces in~\cite[\S3]{pippenger2006topological}. Clearly $S$ and $S'$ are always connected, and produce any connected compact orientable surface with boundary. The topology of such surfaces is determined by their genus and number of boundary components. We describe the limit distribution of these two random variables. 

\begin{theorem*}\label{1}
Let $m = m(n)$ for even $n\in\mathbb{N}$ be such that $1 \ll m \ll n$, and let $r \in [0,1]$ be such that
$$ \frac{\log m}{\log n} \;\;\xrightarrow[\;\;n \to \infty\;\;]{}\;\; r^2 \;\;.$$
The genus $G$ and the number of boundary components $B$ of a random surface, sampled either from $T'(n,m)$ or from $S'(n,m)$, satisfy:
$$ \left[\;\begin{matrix} \displaystyle\frac{1}{\sqrt{\log m}}(B-\log m) \\[+1em] \displaystyle\frac{2}{\sqrt{\log n}}\left(G-\left(\tfrac14 n - \tfrac12\log n\right)\right) \end{matrix}\;\right] \;\;\;\xrightarrow[\;\;n\to\infty\;\;]{\;\;\;\mathcal{D}\;\;\;}\;\;\; \mathcal{N}\left(\,\left[\begin{matrix}0\\0\end{matrix}\right],\left[\begin{matrix}1&-r\,\\-r\,&1\end{matrix}\right]\,\right) 
$$
\end{theorem*}

\begin{corollary*}
The same limit law holds for the genus $G$ and the number of boundary components $B$ of a random surface from $T(n,m)$ or $S(n,m)$, if the anti-correlation coefficient $r$ is smaller than $\sqrt{0.5} \approx 0.71$.
\end{corollary*}

\begin{proof}[Proof (of the corollary)]
For such $r$, boundary edges occur in any two given locations with probability $O(({m}/{n})^2)=o({1}/{n})$, since $m \ll \sqrt{n}$. Hence all the chosen boundary locations are separated by unchosen ones with probability approaching one, because this means avoiding only $O(n)$ adjacent pairs. Conditioning on this event, $S$ reduces to~$S'$, and $T$ reduces to~$T'$.
\end{proof}

\section{Proof}

The gluing is analyzed similarly to the closed case~\cite{gamburd2006poisson, chmutov2016surface}. We randomly label the non-boundary edges by $\{1, \dots, N\}$, uniformly from all $N!$ choices, and independently of the matching that glues them. This extra layer of randomness was originally introduced only for the analysis of the gluing, but here it will come useful also for the study of the boundary. As in~\cite{chmutov2016surface}, the same argument applies to both models, where \mbox{$N=3n$} for a surface in $T'(n,m)$ and $N=n$ for~$S'(n,m)$.

Let $\alpha$ be the order-$N$ permutation that maps every non-boundary edge to the next one clockwise in its polygon, and let $\beta$ map every edge to the one it is glued to. Note that $\alpha$~is uniform of cycle type $(3,3,\dots,3)$ or~$(N)$, depending on the model, and $\beta$ is uniform of type $(2,2,\dots,2)$. Crucially, the gluing map $\beta$ is independent of the map $\alpha$. In the original setting of no boundary, every cycle of $\gamma = \alpha \circ \beta$ corresponds to an internal vertex of the glued surface, and the cycle's length is the number of polygon corners around it. 

The boundary of $S'$ or $T'$ consists of $m$ edges, inserted at a random subset of the $N$ polygon corners. Without loss of generality, we select the corners counterclockwise next to the non-boundary edges labeled~$\{1, \dots, m\}$. Indeed, this choice is independent of the gluing, and uniform from all subsets of $m$ corners. Wherever we insert some boundary edges, the polygon corners no longer meet at an internal vertex, but rather at the boundary vertices that connect those new edges. In terms of $\gamma$, the cycles that contain any of the labels $\{1, \dots, m\}$ now correspond to boundary components, and not to internal vertices.

In conclusion, both the number of internal vertices and the number of boundary components can be read off from the random permutation~$\gamma$. From the cycle types of $\alpha$ and $\beta$ in the model $T'(n,m)$ it follows that $\gamma$~is even iff $N/2$ is even, while $\gamma$ is even iff $N/2$ is odd in~$S'(n,m)$.  

The main result we use from Gamburd and Chmutov--Pittel is that in both models $\gamma$ is approximately uniform. More precisely, up to an $O(1/N)$ total variation distance, $\gamma$~follows the distribution of a uniformly random order-$N$ permutation in the appropriate parity class~\cite[Th~2.2]{chmutov2016surface}. Recall that the \emph{total variation distance} between the distributions of $\gamma$ and~$\gamma'$ is defined as~$\max_E|P(\gamma \in E) - P(\gamma' \in E)|$.

\medskip
A classical method for counting the cycles of a permutation $\gamma$ defines $X_k(\gamma)=1$ if $k$ is the smallest number in its cycle, and $0$ otherwise. Since every cycle contains one minimum, we sum these indicators. Moreover, in our setting we can distinguish between $B$ cycles that correspond to boundaries, and $I$ that correspond to internal vertices, as follows.
$$ B(\gamma) \;=\; \sum\limits_{k=1}^{m}X_k(\gamma) \;\;\;\;\;\;\;\; I(\gamma) \;= \sum\limits_{k=m+1}^{N}X_k(\gamma) $$

Before we study these sums, we correct the parity bias of $\gamma$. Flip a fair coin, and let $\delta = (1\,2) \circ \gamma$ if it comes heads, and $\delta = \gamma$ if tails. Thus $\delta$ is equally likely to be even or odd, and it is $O(1/N)$ distant in total variation from the uniform distribution on order-$N$ permutations. Note that $X_k(\delta)=X_k(\gamma)$ for all $k \neq 2$, while $X_2(\delta)$ is independent of $\gamma$ and uniform in $\{0,1\}$ and therefore might differ from~$X_2(\gamma)$. In conclusion, $|B(\delta)-B(\gamma)| \leq 1$ and $I(\delta)=I(\gamma)$ for $m \geq 2$.

For a perfectly uniform permutation $\varepsilon$, the classical argument says that $X_1(\varepsilon),X_2(\varepsilon),X_3(\varepsilon),\dots,X_N(\varepsilon)$ are mutually independent and $X_k(\varepsilon)=1$ with probability~$1/k$. In short, this follows by randomly throwing the values $N,...,3,2,1$ one by one into an array of size $N$ representing~$\varepsilon$. At each step, the probability that $k$ closes a cycle in $\varepsilon$ is~$1/k$, independently of what has happened previously. See Feller's textbook~\cite[pages 257-258]{feller1968introduction}. It follows that the sums $B(\varepsilon)$ and $I(\varepsilon)$ are independent.

The indicator variables $X_k(\varepsilon)$ are used to show that the distribution of the total number of cycles $S_N = X_1+\dots+X_N$ is approximately normal with asymptotic expected value $\log N$ and variance $\log N$. This requires a variant of the central limit theorem for independent random variables that are not necessarily identically distributed. The following version is sufficient in this case: $(S_N-\mathbb{E}S_N)/\sigma(S_N) \to \mathcal{N}(0,1)$ in distribution as $N \to \infty$, given uniformly bounded summands $|X_k| \leq 1$ with a divergent variance $\sigma^2(S_N) \to \infty$ [\emph{ibid.}, pages 254-255]. 

The Harmonic sum $\sum_{k=1}^N\tfrac1k = (1 + o(1))\log N$ gives the estimate of the expected value and the variance of the total cycle count~$S_N$~[\emph{ibid.}]. By the same estimate, its two parts $B$ and $I$ may be normalized as follows:
$$ \hat{B}(\varepsilon) \;=\; \frac{B(\varepsilon) - \log m}{\sqrt{\log m}} \;\;\;\;\;\;\;\; \hat{I}(\varepsilon) \;=\; \frac{I(\varepsilon) - \log (N/m)}{\sqrt{\log (N/m)}} $$
Since $1 \ll m \ll N$ as $n \to \infty$, these two random variables converge in distribution to independent standard Gaussians.

It follows that also $(\hat{B}(\delta),\hat{I}(\delta))$ converges to a standard 2-dimensional Gaussian distribution, since the total variation distance between $\delta$ and $\varepsilon$ is~$o(1)$ as shown above. In order to go back to $\gamma$ we use Slutsky's Theorem in $\mathbb{R}^2$, which asserts that if a sequence of random variables $Z_n \to Z$ in distribution, and for another sequence $Y_n$ the distance $d(Z_n,Y_n) \to 0$ in probability, then $Y_n \to Z$ in distribution as well \cite[Th~3.1]{billingsley2013convergence}. Therefore, the bivariate normal limit holds also for $(\hat{B}(\gamma),\hat{I}(\gamma))$, whose euclidean distance to $(\hat{B}(\delta),\hat{I}(\delta))$ is at most $1/\sqrt{\log m} = o(1)$.

In particular, the normalized number of boundary components $\hat{B}(\gamma)$ is asymptotically normal as claimed in the theorem.

\medskip
As in previous works, we compute the genus via Euler's characteristic, using the number of internal vertices~$I(\gamma)$. Recall that the Euler characteristic of a closed connected orientable surface of genus $G$ is equal to $2-2G$. For a surface with $B$ boundary components it is $2-2G-B$, such that the genus is invariant under filling boundaries with disks. 

With probability $1-O(1/n)$, the surface $T'(n,m)$ is connected~\cite[\S3]{pippenger2006topological}, and its Euler characteristic is
\begin{align*}
2-2G-B \;&=\; \#\text{faces}-\#\text{edges}+\#\text{vertices} \\
\;&=\; n - (\tfrac32 n + m) + (I+m)    
\end{align*}
For $S'(n,m)$, the terms $n$ and $\frac32 n$ are respectively replaced by $1$ and $\frac{1}{2}n$. In either case,
$$ G \;=\; \tfrac14 n - \tfrac12 (B+I) + O(1) $$

The theorem now follows by a linear change of variables in the 2-Gaussian. The first random variable is $\hat{B}(\gamma)$ as above, and the second one is
\begin{align*}
\frac{G\;-\;\left(\tfrac14 n - \tfrac12\log n\right)}{\tfrac12 \sqrt{\log n}} \;\;&=\; \;-\, \sqrt{\tfrac{\log m}{\log n}} \,\hat{B}(\gamma) \;-\, \sqrt{\tfrac{\log (n/m)}{\log n}} \,\hat{I}(\gamma) \;+\; o(1) \\[0.5em]
\;&=\; \;-\,r\,\hat{B}(\gamma)\;-\, \sqrt{1-r^2}\,\hat{I}(\gamma)\;+\;o(1)    
\end{align*}
with probability $1-o(1)$ as $n \to \infty$. Hence the desired limit distribution is given by the same linear combinations of two standard Gaussians.
\qed

\section{Remarks}

\noindent 
\textbf{(1)}
If $\log m \ll \log n$ then $r=0$ and the limit is a standard two-dimensional Gaussian, so that $B$ and $G$ are ``asymptotically independent''. 

\medskip
\noindent 
\textbf{(2)}
At the other extreme $r=1$, the limiting Gaussian becomes degenerate, supported on the diagonal line: $2B + 4G = n + o(1/\sqrt{\log n})$ with high probability. The theorem actually extends without changes to the case of even larger boundary, $m = \Theta(n)$. The only difference in the proof is that the internal vertices are not asymptotically normal, but they become negligible.

\medskip
\noindent 
\textbf{(3)}
Setting $m=0$ in the random models reduces to the no-boundary case $B=0$, while $m=N$ yields a model of \emph{ribbon graphs} -- disks that are glued along disjoint segments. This latter case reduces to the no-boundary case as well, since $B$~depends on~$G$ deterministically, see~\cite{pippenger2006topological}.

\medskip
\noindent 
\textbf{(4)}
If we fix a positive $m$ and let $n$ grow, then $B$ converges to the discrete distribution of cycles in a random order-$m$ permutation. This is given by the Stirling numbers of the first kind: 
$$ P(B=b) \;\;\xrightarrow[\;\;n \to \infty\;\;]{}\;\; \tfrac{1}{m!}\left[\begin{smallmatrix}m\\b\end{smallmatrix}\right]\; \;\;\;\;\;\;\;\; b \in \{1,\dots,m\} $$
This case requires a more careful treatment of the parity issue. In short, consider the cycles of an order-$n$ uniformly random, say, \emph{even} permutation and delete all elements greater than~$m$. One can show by induction on deletions that the resulting order-$m$ permutation satisfies $\left|P(\text{even})-P(\text{odd})\right| = m(m-1)/n(n-1)$, being uniform within each parity class. It thereby converges in total variation to a uniformly random one, since $m \ll n$. 

\medskip
\noindent 
\textbf{(5)}
The above models $T(n,m)$ and $T'(n,m)$ and these results extend in a straightforward way from triangles to $t$-sided polygons, for any $t \geq 3$. 

\medskip
\noindent 
\textbf{(6)}
It is plausible that other natural variations on these models behave the same. For example, one may randomly \emph{replace} polygon edges by boundary edges, instead of \emph{adding} them. This approach leads to fixed points in the gluing map~$\beta$. When $m$ is small, $\gamma$~may still be approximately uniform, since Larsen and Shalev allow $n^{o(1)}$ fixed points~\cite{larsen2008characters}, though there are other complications to address.

\medskip
\noindent
\textbf{Acknowledgements.} Michael Farber was partially supported by a grant from the Leverhulme Foundation. Chaim Even-Zohar was supported by the Lloyd’s Register Foundation / Alan Turing Institute programme on
Data-Centric Engineering.

We thank the anonymous reviewer for helpful suggestions that greatly improved the final version of the paper.

\bibliographystyle{alpha}
\bibliography{_main}

\appendix 

\end{document}